\documentclass[twocolumn]{autart}    

\makeatletter
\def\ps@copyright{\ps@plain}
\makeatother

\usepackage{graphicx}          
															
\usepackage{amsmath}        
\usepackage{amssymb,latexsym}
\usepackage{cite}            
\usepackage{multicol}         
\usepackage{pdfsync}             
\usepackage[active]{srcltx}  
\usepackage{fixmath}         
\usepackage{subfigure}
\usepackage{epstopdf}
\usepackage{graphicx}

\usepackage[utf8]{inputenc}

\newtheorem{theorem}{Theorem}

\newtheorem{lemma}[theorem]{Lemma}
\newtheorem{property}{Property}[section]  
{\bf}{\it}

\newtheorem{remark}{Remark}[section]{\em}
\newenvironment{proof}{{\it Proof.}}{\hfill $\Box$\\\\}

\newcommand{\norm}[1]{\left\Vert#1\right\Vert}
\newcommand{\abs}[1]{\left\vert#1\right\vert}

\newcommand{\be}{\begin{equation}}
\newcommand{\ee}{\end{equation}}
\newcommand{\beano}{\begin{eqnarray*}}
\newcommand{\eeano}{\end{eqnarray*}}
\newcommand{\ba}{\begin{array}}
\newcommand{\ea}{\end{array}}

\begin{document}

\begin{frontmatter}

\title{Disturbance rejection for classes of nonlinear systems \thanksref{footnoteinfo}} 

\thanks[footnoteinfo]{This paper was not presented at any IFAC 
meeting. Corresponding author S.~Messineo.}

\author[Rome]{Saverio Messineo}\ead{saverio.messineo@fh-salzburg.ac.at}

\address[Rome]{Josef Ressel Centre for Intelligent and Secure Industrial Automation,
Salzburg University of Applied Sciences, Austria}


\begin{abstract} 
This paper addresses the problem of non-adaptive global robust disturbance rejection for two distinct classes of nonlinear systems. The first class, denoted by $\mathcal{C}_1$, consists of nonlinear systems in strict-feedback form, with linear and Hurwitz zero-dynamics (whose states are unavailable for feedback), and \textit{enhanced} - within this work - by \textit{forcing, unmatched, additive disturbances}. Nonlinear tools are herein employed to demonstrate that the proposed control architecture - based on the high-gain paradigm - achieves closed-loop input-to-state stability with respect to the forcing disturbances, along with global asymptotic convergence towards an attractor which can be rendered as small as desired. Then, owing to the established input-to-state stability property, a uniformly bounded control action is additionally embedded within the control architecture.
The additional unit, designed following the sliding-mode paradigm, is aimed at improving the disturbance rejection task, by potentially lowering the required high-gain control expenditure.
The second class of systems, denoted by $\mathcal{C}_2$, is constituted by minimum-phase, uncertain, nonlinear systems with relative degree greater than one, featuring possibly unbounded, with possibly unbounded derivatives, output-dependent nonlinearities, \textit{with matched additive forcing disturbances}. To solve the problem of \textit{output-feedback}, non-adaptive, global robust disturbance rejection for systems within $\mathcal{C}_2$, first, an open-loop observer is employed in lieu of a classic dynamic extension adopted in earlier works, as the latter is no longer implementable due to the presence of unknown forcing disturbances.
Subsequently, the results derived for $\mathcal{C}_1$ are then adapted to $\mathcal{C}_2$, to yield 
an output-feedback dynamic controller 
providing closed-loop global uniform boundedness, along with 
asymptotic regulation towards an attractor which can be rendered as small as desired.

\end{abstract}

\end{frontmatter}     

\section{Introduction}  
In this work, the problem of non-adaptive global robust disturbance rejection is addressed for two distinct classes of nonlinear systems. 

The first class of interest, henceforth denoted by $\mathcal{C}_1$, is constituted by the set of nonlinear systems in strict-feedback form, with linear and Hurwitz zero-dynamics (whose states are not available for control design), and \textit{enhanced} - within the proposed work - by \textit{forcing, unmatched, additive disturbances}. 
From a control-theoretic historical perspective, the centrality of nonlinear systems in strict feedback form stems from the seminal work in~\cite{krstic1995nonlinear}, whereby the backstepping paradigm had been proposed within an adaptive setting. Since then, major theoretical advancements have been achieved 
within several distinct areas of investigation such as state-feedback stabilization~\cite{SEPULCHRE1997979},  output-feedback stabilization~\cite{MAZENC1994119},  time-delay nonlinear systems~\cite{1272269,847117},  adaptive quantized control~\cite{JING20185414}, global finite
time state feedback control~\cite{HUANG2005881}, adaptive event-triggered control~\cite{NING2023111229}, or control of nonholonomic systems~\cite{JIANG2000189}, to name a few. 
From an applicative standpoint, systems in nonlinear strict feedback form are of major relevance as well, as they naturally provide a mathematical tool to model dynamics in engineering domains such as - among others - robot manipulators, hydraulic systems, two-stage chemical reactors, and continuously stirred tank reactors~\cite{zhang2023strict}.

The problem of disturbance rejection for nonlinear strict feedback systems has been addressed in several distinct works, each informed by distinct design paradigms such as the internal model principle~\cite{1211218, Sun2015}, backstepping~\cite{backstepping1,backstepping2,backstepping3}, adaptive control~\cite{704978}, or stochastic control~\cite{7219403}, to simply name a few.
 However, the specific problem of non-adaptive global robust disturbance rejection - by means of partial-state feedback, as the zero-dynamics are not available for control design - for systems within class $\mathcal{C}_1$, has so far remained elusive, to the author's knowledge. 
 This work intends to fill this gap.

From a methodological standpoint, the mathematical tools employed in this work to address systems within class $\mathcal{C}_1$, shall build upon the arguments presented in~\cite[ch. 11]{IsidoriII}.
Indeed, the results within~\cite[ch. 11]{IsidoriII} are extended herein for the purpose of demonstrating, for systems in $\mathcal{C}_1$, closed-loop input-to-state stability with respect to the forcing disturbances, along with global asymptotic convergence towards a residual attractor whose ``size'' can be rendered as small as desired. 
Moreover, the ``high-gain'' control-design philosophy proposed in~\cite[ch. 11]{IsidoriII} is herein extended by embedding - within the proposed control architecture - an additional (or ``external'') unit devoted to enhance the disturbance rejection capabilities of the closed-loop system.
The design of the additional unit is based upon the sliding-mode control paradigm. In so doing, a potentially ``softer'' high-gain control expenditure may now be achieved, as the disturbance rejection task is now addressed in concert by two distinct units, and not any longer by the sole high-gain control effort. In the spirit of the ``external model'' paradigm introduced in~\cite{Serrani_external_model}, and employed in~\cite{AUT_SAV,MESSINEO_AUTOMATICA_2}, the unit devoted to disturbance rejection is embedded within the control architecture so as to preserve the stability of the overall closed-loop system. This is made possible by the fact that the additional unit (by construction, designed as uniformly bounded) is added to a closed-loop system enjoying the input-to-state stability property (which the proposed article shall formally establish).
The establishment of the input-to-state stability property for systems in $\mathcal{C}_1$
is of general value, and it is regarded as a main contribution of this work. Indeed, because of the latter, any suitable design-paradigm (for instance based on the internal model principle~\cite{FrancisWonham75,Francis77}) may be in principle integrated - within the proposed architecture - to pursue disturbance rejection.  
It is furthermore noticed that the provided proof of stability (clearly strongly inspired by~\cite[ch. 11]{IsidoriII}) is inherently inductive and hence - as such - capable to avoid the typical ``curse of dimensionality'' many times affecting 
the mathematical derivations related to systems in strict-feedback form (see, for instance~\cite{krstic1995nonlinear}).

The second class of interest, henceforth denoted by $\mathcal{C}_2$, is constituted by minimum-phase, uncertain, nonlinear systems with relative degree greater than one, featuring possibly unbounded, with possibly unbounded derivatives, output-dependent nonlinearities, \textit{and
enhanced} - within the proposed work -  
\textit{by matched additive forcing disturbances}. Systems within $\mathcal{C}_2$, \textit{but devoid of matched additive forcing disturbances}, have been studied in the seminal work in~\cite{MTII_1993}, whereby major advancements were produced within the area of global, output-feedback, robust (as well as adaptive) stabilization of uncertain nonlinear systems. 
To the author's knowledge, the problem of output-feedback, non-adaptive, global robust disturbance rejection for systems within class $\mathcal{C}_2$ has not been addressed so far. In this work, by building on the derived results for systems within class $\mathcal{C}_1$, the problem of global disturbance rejection for systems in $\mathcal{C}_2$ is addressed, thereby filling the literature gap. 
To begin with, an open-loop observer is proposed in lieu of the dynamic extension employed in~\cite{MTII_1993} and~\cite[ch. 11]{IsidoriII}, as the latter is not any longer implementable because of the presence of the unknown forcing disturbances. Then, by applying the results established for systems in $\mathcal{C}_1$, an output-feedback dynamic controller (equipped with the additional unit devoted to the disturbance rejection task, as per the discussion above pertaining systems in $\mathcal{C}_1$) is derived for systems in $\mathcal{C}_2$. The proposed controller yields closed-loop global uniform boundedness, along with input-to-state stability with respect to the open-loop observer estimation error, and asymptotic regulation towards an attractor which can be rendered as ``small'' as desired.

\section{Class $\mathcal{C}_1$ systems}\label{Stability_analysis_C_1}
The class of $\mathcal{C}_1$ systems is defined by the following properties:

\begin{property}\label{sys_dynamics_2}
The system dynamics can be modelled by the following equations:
\begin{align}\label{sys_generale}
\dot z &= F(\bar \mu)z + G(\xi_1, \mu)\xi_1     \nonumber \\
\dot {\xi}_1 &= H_1(\xi_1, \mu)z +  K_1(\xi_1, \mu)\xi_1 + b_1(\xi_1, \mu) \xi_2  + \bar d_1 \nonumber \\
\dot {\xi}_2 &= H_2(\xi_1, \xi_2, \mu)z + \sum_{i=1}^{2} K_{2i}(\xi_1, \xi_2, \mu) \xi_i  \nonumber \\
& \quad + b_2(\xi_1, \xi_2, \mu)\xi_3 + \bar d_2 \nonumber \\
& \quad ... \nonumber \\
\dot {\xi}_r &= H_r(\xi_1,..., \xi_r, \mu)z + \sum_{i=1}^{r} K_{ri}(\xi_1,..., \xi_r, \mu) \xi_i  \nonumber \\
& \quad + b_r(\xi_1,..., \xi_r, \mu)(u+\Delta_m) + \bar d_r \nonumber \\
u &= u_1 + u_d \nonumber \\
\Delta_d &= u_d + \Delta_m 
\end{align} 
in which $z \in \mathbb{R}^n$, ${\xi}_i \in \mathbb{R}$, $\bar d_i \in \mathbb{R}$, $i=1,...,r$, $u \in \mathbb{R}$, $u_1 \in \mathbb{R}$, $u_d \in \mathbb{R}$, $\Delta_m \in \mathbb{R}$, $\Delta_d \in \mathbb{R}$, $\bar \mu \in {\mathcal {\bar P}} \subset \mathbb{R}^{\bar p}$, $\mu \in {\mathcal P} \subset \mathbb{R}^p$, with ${\mathcal {\bar P}}$ and ${\mathcal P}$ compact, and with $\bar \mu$ and $\mu$ being vectors of unknown terms. Whereas $F(\cdot)$, $G(\cdot)$, $K_1(\cdot)$, $H_i(\cdot)$, $b_i(\cdot)$, $i=1,...,r$, and $K_{ij}(\cdot)$, $i=2,...,r$, $j=2,...,r$, are smooth functions of their arguments. With $u$ being the control input, with $u_d$ to be selected uniformly bounded by design, and with $\bar d_i$, $i=1,...,r$, and $\Delta_m$ being \textit{unknown}, uniformly bounded, disturbances (possibly function of time and/or state). From which it follows that $\Delta_d$ is unknown and uniformly bounded as well.
\end{property}

\begin{property}\label{Zero_dynamics_sys_dynamics_2}
For each $\bar \mu \in {\mathcal {\bar P}}$, the eigenvalues of $F(\bar \mu)$ have negative real part
\end{property}

\begin{property}\label{Well_defined_control_sign_sys_dynamics_2}
There exist real numbers $b_{i0}>0$ such that $b_i(\xi_1,...,\xi_i, \mu) \geq b_{i0}$, for all $1 \leq i \leq r$, $(\xi_1,...,\xi_i) \in \mathbb{R}^i$, and all $\mu \in \mathcal{P}.$
\end{property}

\begin{property}\label{Feedback_info_sys_dynamics_2}
The sole states ${\xi}_i$, $i=1,...,r$ are available for feedback, whereas the state $z$ is not measurable.
\end{property}

\begin{remark}
The dynamics
\begin{align}\label{zero_dyn_fict}
\dot z &= F(\bar \mu)z
\end{align}
are the zero dynamics of system~\eqref{sys_generale} with respect to the ``output'' $\bar y = \xi_1$.
\end{remark}

\begin{remark}\label{remark_compacts}
The uncertainties affecting system~\eqref{sys_generale} have been divided into $\bar \mu \in {\mathcal {\bar P}} \subset \mathbb{R}^{\bar p}$ and $\mu \in {\mathcal P} \subset \mathbb{R}^p$. This is to convey the fact that - for the results that shall follow for systems within class~$\mathcal{C}_1$ - $\bar \mu \in {\mathcal {\bar P}} \subset \mathbb{R}^{\bar p}$ \textit{need} necessarily to be regarded as constant terms to allow the Hurwitz property of $F(\bar \mu)$ to be preserved for any $\bar \mu \in {\mathcal {\bar P}}$. On the contrary, from the derivations that shall follow, it will become clear that the terms $\mu \in {\mathcal P} \subset \mathbb{R}^p$ may be allowed to be functions of time and/or state as well - without affecting the mathematical demonstrations that will follow, along with the derived results - provided they be uniformly bounded, hence ranging within a compact set ${\mathcal P} \subset \mathbb{R}^p$. 
\end{remark}

\begin{remark}\label{remark_disturbances}
As opposed to the functions $F(\cdot)$, $G(\cdot)$, $K_1(\cdot)$, $H_i(\cdot)$, $b_i(\cdot)$, $i=1,...,r$, and $K_{ij}(\cdot)$, $i=2,...,r$, $j=2,...,r$, the uniformly bounded disturbances $\bar d_i \in \mathbb{R}$, $i=1,...,r$ and $\Delta_m$ in system~\eqref{sys_generale}
may be allowed to be function of any state variable, including $z$, as they are not requested to preserve the strict-feedback structure of the system.
\end{remark}

\subsection{Preliminary results}  

\begin{lemma} \label{theorem1}

Consider a system of the form
\begin{align}\label{sys_gen00_A}
\dot z &= \tilde F(\bar \mu)z + \tilde G(y, \mu)y     \nonumber \\
\dot y &= \tilde H(y, \mu)z + \tilde K(y, \mu)y + b(y, \mu) (u+\Delta)  \nonumber \\
u &= u_1 + u_d \nonumber \\
\Delta_1 &= u_d + \Delta
\end{align}
with $z \in \mathbb{R}^n$, $y \in \mathbb{R}$, $u \in \mathbb{R}$,  $u_1 \in \mathbb{R}$, $u_d \in \mathbb{R}$, $\Delta \in \mathbb{R}$, $\Delta_1 \in \mathbb{R}$, $\bar \mu \in {\mathcal {\bar P}} \subset \mathbb{R}^{\bar p}$, $\mu \in {\mathcal P} \subset \mathbb{R}^p$, with ${\mathcal {\bar P}}$ and ${\mathcal P}$ compact, and with $\bar \mu$ and $\mu$ being vectors of unknown terms.
Where the sole state available for feedback is $y$, whereas $z$ is not. The functions $\tilde F(\cdot)$, $\tilde G(\cdot)$, $\tilde H(\cdot)$, $\tilde K(\cdot)$ and $b(\cdot)$, are smooth functions of their arguments, whereas $\Delta$ is a uniformly bounded disturbance (possibly function of time and/or state). With $u$ being the control input, and with $u_d$ to be selected uniformly bounded by design; from which it follows that $\Delta$ be unknown and uniformly bounded as well.

Fix $\delta > 0$ as small as desired. 

Suppose that:

(i) for each $\bar \mu$, the eigenvalues of $F(\bar \mu)$ have negative real part, that is, there exists a symmetric matrix $P(\bar \mu) >0$ (continuously depending on $\bar \mu$) such that: 
 \begin{align}\label{Lyap_equation}
P(\bar \mu) F(\bar \mu) + F(\bar \mu) P(\bar \mu) = -I 
\end{align}

(ii) there exists a real number $b_0 >0$ such that
 \begin{align}\label{lower_bound_control_constant}
b(y, \mu) \geq b_0
\end{align}
for all $y \in \mathbb{R}$ and $\mu \in \mathcal{P}$.

Then, there exist smooth functions $\gamma(y)$ and $\bar u_d(y)$, with $\gamma(y) \geq 0$ and 
with $\bar u_d(y)$ uniformly bounded by design, and there exists $\epsilon >0$ (independent of $\delta$), such that,
if system~\eqref{sys_gen00_A} is controlled by 
 \begin{align}\label{control_input_0}
u &= u_1 + u_d \nonumber \\
u_1 &= - \gamma(y) y \nonumber \\
u_d &= \bar u_d(y) 
\end{align}

then, the following holds: 
\begin{enumerate}

	\item The positive definite function 
\begin{align}\label{Lyap0}
\mathcal{V}= z^TP(\bar\mu)z + y^2
\end{align} 
satisfies 
\begin{align}\label{Lyap00} 
\dot {\mathcal{V}} &\leq -\epsilon (\| z\|^2 + y^2) + \delta \Delta_1^2.
\end{align}
\vspace{1.5mm}

\item System~\eqref{sys_gen00_A}-\eqref{control_input_0} is input-to-state stable with respect to $\Delta_1$.
\vspace{4mm}

\item System~\eqref{sys_gen00_A}-\eqref{control_input_0} possesses an ultimate bound which can be rendered as small as desired. \vspace{4mm}

\item The previous three points hold also in the case in which $u_d \equiv 0$; in such a case, $\Delta_1 = \Delta$ in system~\eqref{sys_gen00_A}-\eqref{control_input_0}. \vspace{4mm} 

\item Consider $p >0$ as large as desired, and $\eta > 0$ as small as desired, and define within the control action~\eqref{control_input_0}
\begin{align}\label{Killer_controlA0}  
\bar u_d(y) = - \frac{ y p^2}{\abs{ y}p + \eta}.
\end{align}
Then there exist $\beta >0$ and $\beta_3 >0$ such that the following holds
\begin{align}\label{LyapA10} 
\dot {\mathcal{V}} &\leq -\frac{\epsilon}{\beta} \mathcal{V}+ 2 \beta_3 \eta.
\end{align}

\end{enumerate}

\end{lemma}
\begin{proof}
Consider the first equation within system~\eqref{sys_gen00_A}, along with the candidate Lyapunov function 
$\mathcal{V}_1= z^TP(\bar \mu)z$. By following identical arguments as in~\cite[pp. 82-83]{IsidoriII}, it is possible to show that $\mathcal{V}_1(z)$ satisfies 
\begin{align}\label{Lyap2}
\dot {\mathcal{V}}_1 &\leq - a  \| z\|^2 + \sigma(|y|)
\end{align} 
where $\sigma(|y|)$ is a $\mathcal{K}_\infty$ function, and $a>0$. Consider now the second equation within system~\eqref{sys_gen00_A}, along with the candidate Lyapunov function $\mathcal{V}_2=y^2$, which is seen to satisfy:  
\begin{align}\label{Lyap3}
\dot {\mathcal{V}}_2 &= 2y\tilde H(y, \mu)z + 2\tilde K(y, \mu)y^2 - 2b(y, \mu)\gamma(y)y^2  \nonumber \\
\quad & + 2b(y, \mu)y\Delta_1.
\end{align} 
Fix $\delta_1$ such that $0<\delta_1<a$, and notice that, because $\mu \in \mathcal{P} \subset \mathbb{R}^p$, with $\mathcal{P}$ compact, then there exist continuous functions $h(y)$ and $k(y)$, such that: $h(y) \geq \left| \tilde H(y, \mu) \right|$, $k(y) \geq \left| \tilde K(y, \mu) \right|$. Moreover, since inequality~\eqref{lower_bound_control_constant} stipulates that  $b(y,\mu)>b_0$, then there exists a continuous function $\bar b(y)$ such that $\bar b(y) \geq b(y,\mu)>b_0$. 
In view of the previous considerations, and by recalling that $\gamma(y) \geq 0$,
an application of Young's inequality to~\eqref{Lyap3}, leads to 
\begin{align}\label{Lyap3_5}  
\dot {\mathcal{V}}_2 &\leq \frac{1}{\delta_1} 2 y^2 h^2(y) + \delta_1 \| z\|^2 + 2k(y)y^2 - 2 b_0 \gamma(y)y^2 \nonumber \\
  & \quad + \frac{1}{\delta} y^2 \bar b(y)^2+ \delta\Delta_1^2.
\end{align} 
By then defining $\mathcal{V} = \mathcal{V}_1 + \mathcal{V}_2$, and by combining inequalities~\eqref{Lyap2} and~\eqref{Lyap3_5}, one obtains:
\begin{align}\label{Lyap4} %
\dot {\mathcal{V}}&\leq -(a - \delta_1)\| z\|^2 -y^2(-bg^2(y)-2\frac{1}{\delta_1} h^2(y) -2k(y) \nonumber \\ \quad &  - \frac{1}{\delta}\bar b(y)^2 + 2 b_0 \gamma(y)) + \delta\Delta_1^2.
\end{align} 
Select $\epsilon \in (0, a-\delta_1)$ so that $a - \delta_1> \epsilon$; then, by following~\cite[p. 83]{IsidoriII}, let $\varphi(y)$ be a continuous function satisfying $\varphi(y)= \varphi(-y)$, with $\varphi(\abs{y})$ positive and non-decreasing, and select $\gamma(y)$ to fulfill:
 \begin{align}\label{bound1} 
 & 2 b_0 \gamma(y) - bg^2(y) - 2 \frac{1}{\delta_1} h^2(y) - 2 k(y) - \frac{1}{\delta}\bar b(y)^2. \nonumber 
\\ \quad & > \varphi^2(y).  
\end{align} 
In view of the previous choices, from~\eqref{Lyap4} and~\eqref{bound1}, it follows that
 \begin{align}\label{Lyap5} 
\dot {\mathcal{V}} &\leq -\epsilon \| z\|^2 -\varphi^2(y) y^2 + \delta\Delta_1^2
\end{align} 
by additionally choosing, without loss of generality, $\varphi(y)$ so that $\varphi^2(y) > \epsilon$, one finally obtains
 \begin{align}\label{Lyap6} 
\dot {\mathcal{V}} &\leq -\epsilon (\| z\|^2 + y^2) + \delta\Delta_1^2
\end{align} 
which proves point~$(1)$ of the theorem, since $\epsilon$ is by construction independent of $\delta$.

In addition, owing to the existence of a number $\beta > 1$ such that 
\begin{align}\label{beta_inequality} 
   z^T P(\bar \mu) z \leq \beta \norm{z}^2 
\end{align}
then, inequality~\eqref{Lyap6} immediately leads to
 \begin{align}\label{Lyap7} 
\dot {\mathcal{V}} &\leq -\frac{\epsilon}{\beta} \mathcal{V}+ \delta\Delta_1^2
\end{align} 
which proves input-to-state stability, hence point~$(2)$ of the theorem.

Moreover, by appealing to the Comparison Lemma~\cite[pp. 102-104]{khalil2002nonlinear}, and by using the fact that $\delta$ can be chosen arbitrarily small, and that $\epsilon$ be independent of $\delta$, 
then it is seen that system~\eqref{sys_gen00_A} possesses an ultimate bound which can be rendered as small as desired: this proves point~$(3)$ of the theorem. 

From the previous derivations, it immediately follows that point~$(4)$ of the theorem holds as well. This is because the sole property that has been exploited in relation to $u_d$ in~\eqref{control_input_0}, is that $u_d$ be uniformly bounded, so that $\Delta_1$ be uniformly bounded as well.

To prove point $(5)$ of the statement, consider again inequality~\eqref{Lyap3}. From the latter, it immediately follows that 
\begin{align}\label{Lyap3_5_sub}  
\dot {\mathcal{V}}_2 &\leq \frac{1}{\delta_1} 2 y^2 h^2(y) + \delta_1 \| z\|^2 + 2k(y)y^2 - 2 b_0 \gamma(y)y^2 \nonumber \\
  & \quad + 2 b(y, \mu) y \Delta_1.
\end{align}
By then choosing $\gamma(y)$ in~\eqref{Lyap3_5_sub} to fulfill:
 \begin{align}\label{bound1} %
 & 2 b_0 \gamma(y) - bg^2(y) - 2 \frac{1}{\delta_1} h^2(y) - 2 k(y)  > \varphi^2(y)  
\end{align}  
with $\gamma(y)^2 > \epsilon$, one obtains
\begin{align}\label{Lyap8} %
\dot {\mathcal{V}} &\leq -\frac{\epsilon}{\beta} \mathcal{V}+ 2b(y, \mu)y\Delta_1.
\end{align} %
Since $\Delta_1 = u_d + \Delta$, and by recalling (from~\eqref{control_input_0}-\eqref{Killer_controlA0}) that
\begin{align}\label{Killer_controlA0} 
u_d = - \frac{ y p^2}{\abs{ y}p + \eta}
\end{align}
then it follows that inequality~\eqref{Lyap8} can be re-written as 
\begin{align}\label{Lyap9} 
\dot {\mathcal{V}} &\leq -\frac{\epsilon}{\beta} \mathcal{V}+ 2 y b(\cdot) \bigg(\Delta - \frac{y p^2}{\abs{y}p + \eta}\bigg).
\end{align}
Moreover, 
it is easy to see that the control in~\eqref{Killer_controlA0} is uniformly bounded, because
\begin{align}\label{Killer_controlAB} 
\frac{ y p^2}{\abs{ y}p + \eta} \leq p.
\end{align}
In view of points $(1)$ to $(4)$ of the theorem, which, in particular, dictate that $b_2(\cdot)$ be a uniformly bounded quantity,
it follows that there exists a constant $\beta_3 >0$ such that $\abs{b_2(\cdot)} \leq \beta_3$. Moreover, since $\Delta$ is uniformly bounded, it is always possible to choose $p>0$ within~\eqref{Killer_controlA} such that $p \geq \abs{\Delta}$. As a result, by indeed choosing $p$ satisfying $p \geq \abs{\Delta}$,
from inequality~\eqref{Lyap9}, it is possible to derive
\begin{align}\label{Lyap9A} 
\dot {\mathcal{V}} &\leq -\frac{\epsilon}{\beta} \mathcal{V}+ 2 \beta_3 \bigg(p\abs{y} - \frac{y^2 p^2}{\abs{y}p + \eta}\bigg) \nonumber \\ & \quad  = -\frac{\epsilon}{\beta} \mathcal{V}+ 2 \beta_3 \bigg(\frac{\abs{y}p\eta}{\abs{y}p + \eta}\bigg).
\end{align}
From the fact that $\frac{mn}{m+n}\leq n$ for any $m \geq 0$ and $n >0$, it follows that $\Big(\frac{\abs{y}p\eta}{\abs{y}p + \eta}\Big) \leq \eta$ within~\eqref{Lyap9A}; as a result, from inequality~\eqref{Lyap9A} 
one finally obtains 
\begin{align}\label{Lyap10} 
\dot {\mathcal{V}} &\leq -\frac{\epsilon}{\beta} \mathcal{V}+ 2 \beta_3 \eta.
\end{align}
By appealing again to the Comparison Lemma in relation to inequality~\eqref{Lyap10}, since $\eta$ can be chosen as small as desired, it is seen that system~\eqref{sys_gen00_A} possesses an ultimate bound which can be rendered as small as desired: this proves point $(5)$ of the theorem, and concludes the proof.
\end{proof}
\begin{lemma}\label{theorem2}
Consider a system of the form
\begin{align}\label{sys_gen00}  
\dot z &= \tilde F(\bar \mu)z + \tilde G(x_1, \mu)x_1     \nonumber \\
\dot x_1 &= \tilde H_1(x_1, \mu)z + \tilde K_1(x_1, \mu)x_1 + b_1(x_1, \mu) y + \bar \Delta   \nonumber \\
\dot y &= H_2 (x_1, y, \mu)z + K_{21} (x_1, y, \mu) x_1 + K_{22}(x_1, y, \mu) y \nonumber \\
& \quad + b_2(x_1, y, \mu) (u+\Delta) + d_2 \nonumber \\
u &= u_1 + u_d \nonumber \\
\Delta_1 &= u_d + \Delta
\end{align}
with $z \in \mathbb{R}^n$, $x_1 \in \mathbb{R}^i$, $y \in \mathbb{R}$, $u \in \mathbb{R}$,  $u_1 \in \mathbb{R}$, $u_d \in \mathbb{R}$, $\bar \Delta \in \mathbb{R}^i$, $\Delta \in \mathbb{R}$, $d_2 \in \mathbb{R}$, $\Delta_1 \in \mathbb{R}$, $\bar \mu \in {\mathcal {\bar P}} \subset \mathbb{R}^{\bar p}$, $\mu \in {\mathcal P} \subset \mathbb{R}^p$, with ${\mathcal {\bar P}}$ and ${\mathcal P}$ compact, and with $\bar \mu$ and $\mu$ being vectors of unknown terms.
Where the sole states available for feedback are $x_1$ and $y$, whereas $z$ is not. 
The functions $\tilde F(\cdot)$, $\tilde G(\cdot)$, $\tilde H_1(\cdot)$, $\tilde K_1(\cdot)$, $b_1(\cdot)$, $H_2(\cdot)$, $K_2(\cdot)$, $K_{22}(\cdot)$, $b_2(\cdot)$ are smooth functions of their arguments, whereas $\bar \Delta$, $\Delta$ and $d_2$ are uniformly bounded disturbances (possibly function of time and/or state). With $u$ being the control input, with $u_d$ to be selected uniformly bounded by design; from which it follows that $\Delta$ be unknown and uniformly bounded as well.

Consider a function $\varphi(\bar \Delta)$ such that $\|\varphi(\bar \Delta)\|^2$ is a class $\mathcal{K}$ function of the argument $\bar \Delta$, and fix $\delta > 0$ as small as desired.

Suppose that: 

(i) there exist a symmetric matrix $P(\bar \mu) >0$ continuously depending on $\bar \mu$, an $i \times i$ matrix $M_1(x_1)$ of smooth functions, non-singular for all $x_1$, a $1 \times i$ vector $\gamma_1(x_1)$ of smooth functions, and there exists $\epsilon >0$ (independent of $\delta$)
such that
\begin{align}\label{Th2_hyp1}
 \|M_1(x_1)x_1\|^2 \geq \underline{\alpha}(\norm{x_1}) \ \ \ \ \text{for all} \  x_1 \in \mathbb{R}^i
\end{align}
for some $\mathcal{K}_\infty$ function $\underline{\alpha}(\cdot)$, and such that the positive definite function
  \begin{align}\label{V1} 
\mathcal{V}_1 &= z^T P(\bar \mu) z + \|M_1(x_1)x_1\|^2
 \end{align}   
satisfies
\begin{align}\label{V1dot}
\begin{pmatrix}
\frac{\partial \mathcal{V}_1}{\partial z} & \frac{\partial \mathcal{V}_1}{\partial x_1}
\end{pmatrix} &
\begin{pmatrix}
\tilde F(\cdot)z + \tilde G(\cdot)x_1   \nonumber \\
\tilde H_1(\cdot)z + \tilde K_1(\cdot)x_1 + b_1(\cdot)\gamma_1(x_1)x_1
\end{pmatrix} \\
& \quad \leq - \epsilon (\|z\|^2 + \|M_1(x_1)x_1\|^2) + \delta \|\varphi(\bar \Delta)\|^2.
\end{align}

(ii) there exists a real number $b_{20}>0$ such that 
\begin{align}\label{Th2_hyp2}
 b_2(x_1, y, \mu) \geq b_{20}
\end{align}
for all $(x_1, y) \in \mathbb{R}^{i+1}$ and all $\mu \in \mathcal{P}$. 

Then there exist smooth functions $\gamma_2(x_1, y)$ and $\bar u_d (\gamma_1(x_1)x_1, y)$, with $\bar u_d (y, \gamma_1(x_1)x_1)$ uniformly bounded by design, such that, 
if system~\eqref{sys_gen00} is controlled by 
 \begin{align}\label{control_input}
u &= u_1 + u_d \nonumber \\
u_1 &= -\gamma_2(x_1, y)[y - \gamma_1(x_1)x_1]\nonumber \\
u_d &= \bar u_d (\gamma_1(x_1)x_1, y)
\end{align}
then, the following applies:
\begin{enumerate}
	\item The positive definite function
 \begin{align}\label{V2}  
\mathcal{V}_2 &= z^T P(\bar \mu) z + \|M_1(x_1)x_1\|^2 + [y - \gamma_1(x_1)x_1]^2
 \end{align} 
satisfies
\begin{align}\label{V2dot0} 
\dot {\mathcal{V}}_2 &\leq - \frac{\epsilon}{2} (\|z\|^2 + \|M_1(x_1)x_1\|^2 + [y - \gamma_1(x_1)x_1]^2) \nonumber \\  
& \quad + \bar \delta \|\bar \varphi(\cdot)\|^2
\end{align}
with 
\begin{align}\label{varphi_def}  
\bar \varphi(\cdot) = \begin{pmatrix}
 \varphi(\bar \Delta)   \\
\bar \Delta \\
\Delta_1 \\
d_2
\end{pmatrix} \nonumber\\
\end{align}
where $\bar \delta >0$ is a quantity that can be made as small as desired. 
\vspace{4mm}

\item System~\eqref{sys_gen00}-\eqref{control_input} is input-to-state stable with respect to $\bar \Delta, \Delta_1$ and $d_2$. \vspace{4mm}

\item System~\eqref{sys_gen00}-\eqref{control_input} possesses an ultimate bound which can be rendered as small as desired.  
\vspace{2mm}

\item The previous three points hold also in the case in which $u_d \equiv 0$; in such a case, $\Delta_1 = \Delta$ in system~\eqref{sys_gen00}-\eqref{control_input}. \vspace{2mm}

\item Consider $p >0$ as large as desired, and $\eta > 0$ as small as desired, and define within the control action~\eqref{control_input}
\begin{align}\label{Killer_controlA} 
\bar u_d(y) = - \frac{ (y - \gamma_1(x_1)x_1) p^2}{\abs{ y - \gamma_1(x_1)x_1}p + \eta}.
\end{align}
Then there exist $\beta >0$ and $\beta_2 >0$ such that the following holds
\begin{align}\label{Comparison_lemma_TH2BIS30} 
\dot {\mathcal{V}} &\leq -\frac{\epsilon}{\beta} \mathcal{V}+ 2 \beta_2 \eta.
\end{align}

\end{enumerate}

\end{lemma}

\begin{proof}
Define $\tilde y = y - \gamma_1(x_1)x_1$, then, from~\eqref{sys_gen00}-\eqref{V1dot}-\eqref{control_input} one obtains:
\begin{align}\label{V2dot1}  
\dot {\mathcal{V}}_2  &\leq  - \epsilon (\|z\|^2 + \|M_1(x_1)x_1\|^2) + \delta \|\varphi(\cdot)\|^2  \nonumber \\  & \quad + \tilde y \frac{\partial U(x_1)}{\partial x_1}b_1(\cdot) + 2 \tilde y (\dot y - \gamma_1(x_1) \dot x_1 - x_1^T\frac{\partial \gamma_1^T}{\partial x_1} \dot x_1) \nonumber \\  & \quad + 2 \tilde y (b_2(\cdot) \Delta_1 + d_2 + \tilde \gamma_1(x_1) \bar \Delta)
\end{align} 
with $U(x_1) = \|M_1(x_1)x_1\|^2$ and $\tilde \gamma_1 = \gamma_1(x_1) + x_1^T \frac{\partial \gamma_1^T}{\partial x_1}$. By observing that there exist a matrix of smooth functions satisfying $\frac{\partial U(x_1)}{\partial x_1} = 2 x_1^T W(x_1)$~\cite[p. 85]{IsidoriII}, and by defining, as in~\cite[p. 86]{IsidoriII}:
\begin{align}\label{definitions_Isidori2}  
A(x_1, y, \mu) &= H_2(x_1, y, \mu)- \tilde \gamma_1(x_1) H_1(x_1, \mu) \nonumber \\
B(x_1, y, \mu) &= b_1^T(x_1, \mu) W^T(x_1) + K_{21}(x_1, y, \mu) \nonumber \\
               & \quad - \tilde \gamma_1 (x_1) \Big(K_1(x_1, \mu) + b_1(x_1, \mu) \gamma_1(x_1) \Big) \nonumber \\
               & \quad + K_{22}(x_1, y, \mu) \gamma_1(x_1) \nonumber \\
C(x_1, y, \mu) &= K_{22}(x_1, y, \mu) - \tilde \gamma_1 (x_1) b_1(x_1, \mu)							
\end{align} 
from~\eqref{V2dot1}, after some algebra, one obtains
\begin{align}\label{V2dot2}  
\dot {\mathcal{V}}_2  &\leq  - \epsilon (\|z\|^2 + \|M_1(x_1)x_1\|^2) + \delta \|\varphi(\cdot)\|^2  \nonumber \\
 & \quad + 2 \tilde y \Big( A(\cdot)z + B(\cdot)x_1 + C(\cdot) \tilde y + b_2(\cdot) u_1 \Big) \nonumber \\
& \quad + 2 \tilde y (b_2(\cdot) \Delta_1 + d_2 + \tilde \gamma_1(x_1) \bar \Delta).
\end{align} 
By then selecting the control action as $u_1 = - \gamma_2(x_1, y)\tilde y$, the previous inequality can be rewritten as 
\begin{align}\label{V2dot3} 
\dot {\mathcal{V}}_2  &\leq  - \epsilon (\|z\|^2 + \|M_1(x_1)x_1\|^2) + \delta \|\varphi(\cdot)\|^2  \nonumber \\
 & \quad + 2 \tilde y \Big( A(\cdot)z + B(\cdot)x_1 + C(\cdot) \tilde y - b_2(\cdot) \gamma_2(x_1, y)\tilde y \Big) \nonumber \\
& \quad + 2 \tilde y (b_2(\cdot) \Delta_1 + d_2 + \tilde \gamma_1(x_1) \bar \Delta)
\end{align}
Fix $\delta_1, \delta_2$ and $\delta_3$ positive, and as small as desired; then, from~\eqref{V2dot3}, by noticing that  
\begin{align}\label{Y1}
2 \tilde y& (b_2(\cdot) \Delta_1 + d_2 + \tilde \gamma_1(x_1) \bar \Delta) \leq 2 |\tilde y| |\bar b_2(x_1, y)|  \|\Delta_1\| \nonumber \\
 & \quad + 2 |\tilde y| |d_2| + 2 |\tilde y| \|\tilde \gamma_1 (\cdot)\| \| \bar \Delta\| \leq \frac{1}{\delta_1} \tilde y^2 \bar b_2^2(\cdot) + \delta_1 \| \Delta_1\|^2  \nonumber \\
& \quad  + \frac{1}{\delta_2} \tilde y^2 + \delta_2 |d_2|^2 + \frac{1}{\delta_3} \tilde y^2 \| \tilde \gamma_1 (\cdot)\|^2 + \delta_3  \| \bar \Delta \|^2
\end{align}
inequality~\eqref{V2dot3} can be specialized as 
\begin{align}\label{V2dot4}  
\dot {\mathcal{V}}_2  &\leq  - \epsilon (\|z\|^2 + \|M_1(x_1)x_1\|^2) + 2 \tilde y \big( A(\cdot)z + B(\cdot)x_1 \big)\nonumber \\
& \quad + 2 \tilde y^2 \big(C(\cdot)  - b_2(\cdot) \gamma_2(\cdot) + \frac{1}{2 \delta_1} \bar b_2^2(\cdot) +   \frac{1}{2 \delta_2}  \nonumber \\
& \quad +  \frac{1}{2 \delta_3}  \| \tilde \gamma_1 (\cdot)\|^2 \big)  + \delta \|\varphi(\cdot)\|^2  + \delta_1 \| \Delta_1\|^2 + \delta_2 |d_2|^2  \nonumber \\ 
& \quad + \delta_3  \| \bar \Delta \|^2.
\end{align}
From~\eqref{V2dot4}, one obtains
\begin{align}\label{V2dot5}  
\dot {\mathcal{V}}_2  &\leq  
         \begin{pmatrix}
           z \\
           x_1 \\
           \tilde y
         \end{pmatrix}^T 
				Q(x_1, y, \mu)
				\begin{pmatrix}
           z \\
           x_1 \\
           \tilde y
         \end{pmatrix}
				+ \delta \|\varphi(\cdot)\|^2  + \delta_1 \| \Delta_1\|^2  \nonumber \\
& \quad  + \delta_2 |d_2|^2  + \delta_3  \| \bar \Delta \|^2
\end{align}
where 
\begin{align}\label{Q} 
Q(\cdot)= 
&\begin{pmatrix}
           \epsilon I & 0 & -A^T(\cdot) \\
            0 & \epsilon M_1^T(\cdot)M_1(\cdot) & -B^T(\cdot)\\
           -A(\cdot) & -B(\cdot) & 2 [b_2(\cdot)\gamma_2(\cdot)- \bar C(\cdot)]
         \end{pmatrix}
\end{align}
and 
\begin{align}\label{hatC}  
& \quad \bar C(\cdot) =  C(\cdot)  + \frac{1}{2 \delta_1} \bar b_2^2(\cdot) +   \frac{1}{2 \delta_2} +  \frac{1}{2 \delta_3}  \| \tilde \gamma_1 (\cdot)\|^2.   
\end{align}
By then defining   
\begin{align}\label{compact1} 
\bar \delta = \max\{\delta, \delta_1, \delta_2, \delta_3\}
\end{align}
from inequality~\eqref{V2dot5} it is possible to derive:
\begin{align}\label{V2dot6}  
\dot {\mathcal{V}}_2  &\leq  
         \begin{pmatrix}
           z \\
           x_1 \\
           \tilde y
         \end{pmatrix}^T 
				Q(\cdot)
				\begin{pmatrix}
           z \\
           x_1 \\
           \tilde y
         \end{pmatrix}
				+ \bar \delta \|\bar \varphi(\cdot)\|^2 
\end{align}
where $\bar \varphi(\cdot)$ is defined in~\eqref{varphi_def}, and where $\bar \delta \|\bar \varphi(\cdot)\|^2$ is a function of class $\mathcal{K}$ by construction.

By then using standard arguments (see~\cite[p. 81]{IsidoriII}), by definition of $Q(\cdot)$ in~\eqref{Q}, 
it follows that there exists $\gamma_2(x_1, y)$ in~\eqref{Q} such that:
\begin{align}\label{Q2} 
Q(\cdot) > \frac{\epsilon}{2}
&\begin{pmatrix}
           I & 0 & 0 \\
            0 & M_1^T(x_1)M_1(x_1) & 0\\
           0 & 0 & 1
         \end{pmatrix}.
\end{align}
Hence, by using~\eqref{Q2} and by recalling that $\tilde y = y - \gamma_1(x_1)x_1$, from inequality~\eqref{V2dot6} one obtains
\begin{align}\label{V2dot0_thm} 
\dot {\mathcal{V}}_2 &\leq - \frac{\epsilon}{2} (\|z\|^2 + \|M_1(x_1)x_1\|^2 + [y - \gamma_1(x_1)x_1]^2) \nonumber \\  
& \quad + \bar \delta \|\bar \varphi(\cdot)\|^2.
\end{align}
Since $\delta$, $\delta_1, \delta_2$ and $\delta_3$ are chosen as small as desired, then it follows that $\bar \delta$ defined in~\eqref{compact1} can be chosen as small as desired as well. This fact, along with inequality~\eqref{V2dot0_thm}, proves point~$(1)$ of the theorem. 


Moreover, because there exists $\beta > 1$ such that 
\begin{align}\label{beta_inequality} 
   z^T P z \leq \beta \norm{z}^2 
\end{align}
then, from~\eqref{V2dot0} and~\eqref{beta_inequality}, one obtains
\begin{align}\label{V2dot7}  
\dot {\mathcal{V}}_2 &\leq - \frac{\epsilon}{2 \beta} {\mathcal{V}}_2  + \bar \delta \|\bar \varphi(\cdot)\|^2
\end{align}
which shows that system~\eqref{sys_gen00} is input-to-state stable with respect to the disturbances $\bar \Delta, \Delta_1$ and $d_2$, and hence proves point~$(2)$ of the theorem.

Moreover, since it has been shown that $\bar \delta$ in~\eqref{V2dot7} can be rendered as small as desired, since $\|\bar \varphi(\cdot)\|^2$ is uniformly bounded and not a function of $\bar \delta$, and since $\epsilon$ is independent of $\bar \delta$ by construction, then, by applying the Comparison Lemma to inequality~\eqref{V2dot7}, it can be concluded that system~\eqref{sys_gen00} possesses also an ultimate bound which can be rendered as small as desired, proving point~$(3)$ of the theorem.  From the previous derivations, it immediately follows that point~$(4)$ of the theorem holds as well. This is because the sole property that has been exploited in relation to $u_d$ in~\eqref{control_input}, is that $u_d$ be uniformly bounded, so that $\Delta_1$ be uniformly bounded as well.

To prove point $(5)$ of the statement, consider again inequality~\eqref{V2dot2}. From the latter, it immediately follows that 
\begin{align}\label{V2dot3BIS} 
\dot {\mathcal{V}}_2  &\leq  - \epsilon (\|z\|^2 + \|M_1(x_1)x_1\|^2) + \delta \|\varphi(\cdot)\|^2  \nonumber \\
 & \quad + 2 \tilde y \Big( A(\cdot)z + B(\cdot)x_1 + C(\cdot) \tilde y - b_2(\cdot) \gamma_2(x_1, y)\tilde y \Big) \nonumber \\
& \quad + 2 \tilde y (d_2 + \tilde \gamma_1(x_1) \bar \Delta) + 2 \tilde y b_2(\cdot) \Delta_1
\end{align}
Moreover notice that, from from inequality~\eqref{Y1}, it is possible to obtain
\begin{align}\label{Y1BIS}
2 \tilde y (d_2 + \tilde \gamma_1(x_1) \bar \Delta) &\leq + \frac{1}{\delta_2} \tilde y^2 + \delta_2 |d_2|^2 + \frac{1}{\delta_3} \tilde y^2 \| \tilde \gamma_1 (\cdot)\|^2  \nonumber \\
& \quad   + \delta_3  \| \bar \Delta \|^2.
\end{align} 
By then using~\eqref{Y1BIS} in combination with~\eqref{V2dot3BIS}, one obtains
\begin{align}\label{V2dot4BIS}  
\dot {\mathcal{V}}_2  &\leq  - \epsilon (\|z\|^2 + \|M_1(x_1)x_1\|^2) + 2 \tilde y \big( A(\cdot)z + B(\cdot)x_1 \big)\nonumber \\
& \quad + 2 \tilde y^2 \big(C(\cdot)  - b_2(\cdot) \gamma_2(\cdot) +   \frac{1}{2 \delta_2} +  \frac{1}{2 \delta_3}  \| \tilde \gamma_1 (\cdot)\|^2 \big) \nonumber \\
& \quad   + \delta \|\varphi(\cdot)\|^2 + \delta_2 |d_2|^2 + \delta_3  \| \bar \Delta \|^2 + 2 \tilde y b_2(\cdot) \Delta_1
\end{align}
which, in turn, leads to
\begin{align}\label{V2dot5BIS}  
\dot {\mathcal{V}}_2  &\leq  
         \begin{pmatrix}
           z \\
           x_1 \\
           \tilde y
         \end{pmatrix}^T 
				\bar Q(x_1, y, \mu)
				\begin{pmatrix}
           z \\
           x_1 \\
           \tilde y
         \end{pmatrix}
				+ \delta \|\varphi(\cdot)\|^2  + \delta_2 |d_2|^2 \nonumber \\
& \quad  + \delta_3  \| \bar \Delta \|^2 + 2 \tilde y b_2(\cdot) \Delta_1
\end{align}
with
\begin{align}\label{QBIS} 
&\bar Q(\cdot)= \nonumber \\
&\begin{pmatrix}
           \epsilon I & 0 & -A^T(\cdot) \\
            0 & \epsilon M_1^T(\cdot)M_1(\cdot) & -B^T(\cdot)\\
           -A(\cdot) & -B(\cdot) & 2 [b_2(\cdot)\gamma_2(\cdot)- \bar C_2(\cdot)]
         \end{pmatrix}
\end{align}
and 
\begin{align}\label{hatC}  
& \quad \bar C_2(\cdot) =  \bar C(\cdot)  - \frac{1}{2 \delta_1} \bar b_2^2(\cdot).   
\end{align}
By defining 
\begin{align}\label{compact1BIS} 
\bar \delta_2 = \max\{\delta, \delta_2, \delta_3\}
\end{align}
where $\bar \delta_2 \leq \bar \delta$ (with $\bar \delta$ defined in~\eqref{compact1}) by construction, one obtains
\begin{align}\label{V2dot6BIS}  
\dot {\mathcal{V}}_2  &\leq  
         \begin{pmatrix}
           z \\
           x_1 \\
           \tilde y
         \end{pmatrix}^T 
				\bar Q(\cdot)
				\begin{pmatrix}
           z \\
           x_1 \\
           \tilde y
         \end{pmatrix}
				+ \bar \delta_2 \|\bar \varphi_2(\cdot)\|^2 + 2 \tilde y b_2(\cdot) \Delta_1
\end{align}
with
\begin{align}\label{compact_statement_th2} 
\bar \varphi_2 (\cdot) = \begin{pmatrix}
            \varphi(\cdot)\\
						d_2\\
           \bar \Delta
         \end{pmatrix}.    
\end{align}
Then, by building on the previous considerations, because $\bar C_2(\cdot) < \bar C(\cdot)$, it is seen that the previously-derived function $\gamma_2(x_1, y)$, now employed in~\eqref{QBIS}, is also capable to guarantee that: 
\begin{align}\label{Q2BIS} 
\bar Q(\cdot) > \frac{\epsilon}{2}
&\begin{pmatrix}
           I & 0 & 0 \\
            0 & M_1^T(x_1)M_1(x_1) & 0\\
           0 & 0 & 1
         \end{pmatrix}.
\end{align}
As a result, from~\eqref{V2dot6BIS} and~\eqref{Q2BIS}, one obtains
\begin{align}\label{Comparison_lemma_TH2BIS} 
\dot {\mathcal{V}}_2 &\leq - \frac{\epsilon}{2 \beta} \mathcal{V}_2 + \bar \delta_2 \|\bar \varphi_2(\cdot)\|^2 + 2 \tilde y b_2(\cdot) \Delta_1.
\end{align}
Since $\Delta_1 = u_d + \Delta$, and by recalling (from~\eqref{control_input}-\eqref{Killer_controlA}) that
\begin{align}\label{Killer_control}  
u_d = -\frac{\tilde y p^2}{\abs{\tilde y}p + \eta}
\end{align}
then it follows that inequality~\eqref{Comparison_lemma_TH2BIS} can then be re-written as
\begin{align}\label{Comparison_lemma_TH2BIS2}
\dot {\mathcal{V}}_2 &\leq - \frac{\epsilon}{2 \beta} \mathcal{V}_2 + \bar \delta_2 \|\bar \varphi_2(\cdot)\|^2 + 2 \tilde y b_2(\cdot) (\Delta - \frac{\tilde y p^2}{\abs{\tilde y}p + \eta}).
\end{align}
In view of points $(1)$ to $(4)$ of the theorem, which, in particular, dictate that $b_2(\cdot)$ be a uniformly bounded quantity,
it follows that there exists a constant $\beta_2 >0$ such that $2 \abs{b_2(\cdot)} \leq \beta_2$. Moreover, since $\Delta$ is uniformly bounded, it is always possible to choose $p$ in~\eqref{Killer_control} such that $p \geq \abs{\Delta}$. As a result, from inequality~\eqref{Comparison_lemma_TH2BIS2}, it is possible to write (by following the same steps as in Lemma~\ref{theorem1}, which led from inequality~\eqref{Lyap9} to~\eqref{Lyap10})
\begin{align}\label{Comparison_lemma_TH2BIS3}
\dot {\mathcal{V}}_2 &\leq - \frac{\epsilon}{2 \beta} \mathcal{V}_2 + \bar \delta_2 \|\bar \varphi_2(\cdot)\|^2 + \beta_2 \eta.
\end{align}
Because $\bar \delta$ (defined in~\eqref{compact1}) can be chosen as small as desired, then $\bar \delta_2$ (defined in~\eqref{compact1BIS}), by construction, can be chosen as small as desired as well.
Moreover, it is noted that $\eta$ can be chosen as small as desired as well, whereas $\|\bar \varphi_2(\cdot)\|^2$ is a uniformly bounded function which does not depend on $\bar \delta_2$, and $\epsilon$ is independent of $\bar \delta_2$ by construction. Hence, by applying the Comparison Lemma to inequality~\eqref{Comparison_lemma_TH2BIS3}, it follows that system~\eqref{sys_gen00} possesses an ultimate bound which can be rendered as small as desired. This concludes the proof.
\end{proof}

\subsection{Closed-loop stability analysis} 
\begin{theorem}\label{theorem_recursion}
Consider system~\eqref{sys_generale}, and assume that Properties~\ref{sys_dynamics_2}-\ref{Zero_dynamics_sys_dynamics_2}-\ref{Well_defined_control_sign_sys_dynamics_2}-\ref{Feedback_info_sys_dynamics_2} hold. Define in system~\eqref{sys_generale}
\begin{align}\label{Compact_coordinates} 
\bar x_1 &=  (\xi_1,...,\xi_{r-1})^T. 
\end{align}
Then, there exists a smooth feedback law for system~\eqref{sys_generale}
\begin{align}\label{control_input_sys_generale}
u  = -\bar\gamma_2(\bar x_1, {\xi}_r)[{ \xi}_r - \bar \gamma_1(\bar x_1)\bar x_1] -\frac{({\xi}_r - \bar \gamma_1(\bar x_1)\bar x_1) p^2}{\abs{{ \xi}_r - \bar \gamma_1(\bar x_1)\bar x_1}p + \eta}
\end{align}
with $\bar \gamma_1(\cdot)$ and $\bar \gamma_2(\cdot)$ smooth functions of their arguments, with $p >0$ large enough, and with $\eta >0$ as small as desired, such that: 
\begin{enumerate}
	\item the closed-loop trajectories are input-to-state stable with respect to $\bar d_i$, $i=1,...,r$, and $\Delta_d$ \vspace{2mm}
	\item the closed-loop trajectories possess a ultimate bound, which can be made as small as desired.
\end{enumerate}

\end{theorem}

\begin{proof}
The proof proceeds by induction on $r$. 

Suppose $r=1$ within system~\eqref{sys_generale}. In such a case, system~\eqref{sys_generale} can be written in the form of system~\eqref{sys_gen00_A}, and it is easy to see that the hypotheses and results of Lemma~\ref{theorem1} immediately apply. Hence points $(1)$ and $(2)$ of the theorem are fulfilled for $r=1$ in system~\eqref{sys_generale}.

Suppose $r=2$ in system~\eqref{sys_generale}. Rewrite the latter in the form of system~\eqref{sys_gen00} by defining $\xi_1=x_1$, $\xi_2=y$, $\bar d_1 = \bar \Delta$, $\Delta_d = \Delta_1$ and $\bar d_r = d_2$. We then observe that because of the results of Lemma~\ref{theorem1}, the hypotheses of Lemma~\ref{theorem2} are satisfied with $\norm{M_1(x_1)x_1}^2 = \xi_1^2$ in~\eqref{V1}, $\gamma_1(x_1)= \gamma_1 (\xi_1)$ in~\eqref{V1dot}, and with inequality~\eqref{Th2_hyp2} fulfilled because Property~\ref{Well_defined_control_sign_sys_dynamics_2} holds by assumption.
As a consequence, the results of Lemma~\ref{theorem1} apply to system~\eqref{sys_generale} whenever $r=2$. Hence points $(1)$ and $(2)$ of the theorem are fulfilled for $r=2$ in system~\eqref{sys_generale}.

Consider again system~\eqref{sys_generale}, and rewrite it in the form of system~\eqref{sys_gen00} by setting  $(\xi_1,...,\xi_i)^T=x_1$, $\xi_r = y$, $(d_1,...,d_i)^T= \bar \Delta$, $\Delta_d = \Delta_1$ and $\bar d_r = d_2$, with $r= i+1$ by construction. Select $i=1$, hence $r=2$ (which is the case consider above): because  
\begin{align}\label{buona_definizione_1}  
 &\|M_1(x_1)x_1\|^2 + [\xi_2 - \gamma_1(x_1)x_1]^2 = \nonumber \\
& \quad \norm{ \begin{pmatrix}
           M_1(x_1)x_1 & 0 \\
           - \gamma_1(x_1) & 1
         \end{pmatrix}     \begin{pmatrix}
           x_1  \\
           \xi_2
         \end{pmatrix} }^2  
 \end{align} 
because point (1) of Lemma~\ref{theorem2} applies, and because
\begin{align}\label{buona_definizione_2}  
 &-\gamma_2(x_1,\xi_2)[\xi_2-\gamma_1(x_1)x_1] = 
\nonumber \\
& \quad \begin{pmatrix} \gamma_2(x_1,\xi_2)\gamma_1(x_1) & -\gamma_2(x_1,\xi_2) \end{pmatrix}  \begin{pmatrix} 
                              x_1 \\ 
															\xi_{i+1}
      \end{pmatrix}
 \end{align} 
 then it is readily seen that the hypotheses of Lemma~\ref{theorem2}, hence its results, hold also for the case $i=2$, hence $r=3$. 
 Hence points $(1)$ and $(2)$ of the theorem are fulfilled for $r=3$ in system~\eqref{sys_generale}.

 Assume the hypotheses of Lemma~\ref{theorem2} hold for $i \geq 2$, hence $r \geq 3$: because 
\begin{align}\label{buona_definizione_1}  
 &\|M_1(x_1)x_1\|^2 + [\xi_{i+1} - \gamma_1(x_1)x_1]^2 = \nonumber \\
& \quad \norm{ \begin{pmatrix}
           M_1(x_1)x_1 & 0 \\
           - \gamma_1(x_1) & 1
         \end{pmatrix}     \begin{pmatrix}
           x_1  \\
           \xi_{i+1}
         \end{pmatrix} }^2  
 \end{align}
 because Lemma~\ref{theorem2} point (1) applies, and because
\begin{align}\label{buona_definizione_2}  
 &-\gamma_2(x_1,\xi_{i+1})[\xi_{i+1}-\gamma_1(x_1)x_1] = 
\nonumber \\
& \quad \begin{pmatrix} \gamma_2(x_1,\xi_{i+1})\gamma_1(x_1) & -\gamma_2(x_1,\xi_{i+1}) \end{pmatrix}  \begin{pmatrix} 
                              x_1 \\ 
															\xi_{i+1}
      \end{pmatrix}
 \end{align} 
 then it is readily seen that the hypotheses of Lemma~\ref{theorem2}, hence its results hold also for the case $i+1$, hence $r+1$. That is, points $(1)$ and $(2)$ of the theorem are fulfilled for the case $r+1$ in system~\eqref{sys_generale}: this proves the result by induction and concludes the proof.
 \end{proof}
\section{Class $\mathcal{C}_2$ systems}\label{Stability_analysis_C_2}
The class of $\mathcal{C}_2$ systems is defined by the following properties:

\begin{property}\label{sys_dynamics}
The system dynamics can be modelled by the following equations:
\begin{align}\label{sys1}
\dot x &= F(\mu)x + G(y, \mu)y + \bar g(\mu) \Gamma (y)(u+\Delta)    \nonumber \\
\dot y &= H(\mu)x + K(y, \mu)y  
\end{align}
in which $x \in \mathbb{R}^n$, $y \in \mathbb{R}$, $u \in \mathbb{R}$, $\Delta \in \mathbb{R}$, $\mu \in {\mathcal P} \subset \mathbb{R}^p$, with ${\mathcal P}$ compact, and $\mu$ being a vector of unknown parameters.
Whereas $F(\cdot)$, $G(\cdot)$, $\bar g(\cdot)$, $\Gamma(\cdot)$, $H(\cdot)$, $K(\cdot)$ are smooth functions of their respective arguments. With $u$ being the control input, $\Gamma(\cdot)$ known,
and $\Delta$ an \textit{unknown}, uniformly bounded, disturbance (possibly function of time and/or state). Moreover, there exists $\Gamma_0 >0$ such that $\Gamma(y) \geq \Gamma_0$ for any $y \in \mathbb{R}$.
\end{property}

\begin{property}\label{min_phase_assumption}
System~\eqref{sys1} possesses asymptotically stable zero-dynamics with respect to $y$. 
\end{property}

\begin{property}\label{relative_degree}
There exists $b_0 > 0$ such that system~\eqref{sys1} has a well-defined relative degree $r \geq 2$, additionally fulfilling:
\begin{align}\label{well_defined_sign}
H(\mu) F^{r-2}(\mu) \bar g(\mu) & \geq b_0  
\end{align}
for any  $\mu \in {\mathcal P}$.
\end{property}

\begin{property}\label{measures}
The sole ``output'' $y$ is available for feedback, whereas the state $x$ is not measurable.
\end{property}

It turns out~\cite[ch. 11.3]{IsidoriII} that systems fulfilling Properties~\ref{sys_dynamics}-\ref{min_phase_assumption}-\ref{relative_degree}-\ref{measures}, fulfill also the following additional property: 
\begin{property}\label{Dyn_extension_old}
Consider the ``dynamic extension''
\begin{align}\label{extension_old}
\dot \xi_2 &= - \lambda_1  \xi_2  +  \xi_3  \nonumber \\
\dot \xi_3 &= - \lambda_2  \xi_3  +  \xi_4  \nonumber \\ 
 & \ ... \nonumber \\ 
\dot \xi_{r} &= - \lambda_{r-1}  \xi_r  +  \Gamma(y)(u + \Delta)
\end{align}
with $\lambda_i >0$, for $i = 1,..., r-1$. Define $\xi =(\xi_2,...,\xi_r)^T$; then 
there exist a constant $\bar{b}_0 >0$, smooth functions $b(\mu)$, $D(\mu)$ and $d(\mu)$, with $b(\mu) > \bar{b}_0$, 
and a change of coordinates 
\begin{align}\label{change_of_coordinates}
z &= x - D(\mu) \xi - \frac{1}{b(\mu)}d(\mu)y
\end{align}
such that the extended system~\eqref{sys1}-\eqref{extension_old} - within the new coordinates - can be written as 
 \begin{align}\label{extended_sys1_old_version}
\dot z &= \tilde F(\mu)z + \tilde G(y, \mu)y     \nonumber \\
\dot y &= \tilde H(\mu)z + \tilde K(y, \mu)y + b(\mu) \xi_2   \nonumber \\
\dot \xi_2 &= - \lambda_1  \xi_2  +  \xi_3  \nonumber \\
\dot \xi_3 &= - \lambda_2  \xi_3  +  \xi_4  \nonumber \\ 
 & \ ... \nonumber \\ 
\dot \xi_{r} &= - \lambda_{r-1}  \xi_r  +  \Gamma(y)(u + \Delta)
\end{align}
in which $\tilde F(\cdot)$, $\tilde G(\cdot)$, $\tilde H(\cdot)$, $\tilde K(\cdot)$ and $b(\cdot)$ are smooth functions of their respective arguments, and $\tilde F(\mu)$ is Hurwitz for any $\mu \in {\mathcal P}$.
\end{property}

\subsection{Preliminary results}
Within the disturbance-free setting in~\cite[ch. 11.3]{IsidoriII} (in which there is indeed no unknown disturbance $\Delta$) the control design for systems fulfilling Properties~\ref{sys_dynamics}-\ref{min_phase_assumption}-\ref{relative_degree}-\ref{measures}, proceeds by employing the ``dynamic extension''~\eqref{extension_old}. 
However, within this study, the disturbance $\Delta$ is, in fact, non-zero, and, because $\Delta$ is \textit{unknown} by definition, then the dynamic extension~\eqref{extension_old} can not be implemented.     As a consequence, in lieu of~\eqref{extension_old}, one might consider 
 the following ``open-loop observer'' for~\eqref{extension_old}
 \begin{align}\label{observer_wrong}
\dot {\hat \xi}_2 &= - \lambda_1  {\hat \xi}_2  +  {\hat \xi}_3  \nonumber \\
\dot {\hat \xi}_3 &= - \lambda_2  {\hat \xi}_3  +  {\hat \xi}_4  \nonumber \\ 
 & \ ... \nonumber \\ 
\dot {\hat \xi}_r &= - \lambda_{r-1}  {\hat \xi}_r  +  \Gamma(y)u. 
\end{align}
Indeed, by defining ${\tilde \xi}_i = \xi_i - {\hat \xi}_i$, and by considering~\eqref{extension_old}-\eqref{observer_wrong}, the resulting observer-error dynamics would read as: 
\begin{align}\label{observer_error_wrong}
\dot {\tilde \xi}_2 &= - \lambda_1  {\tilde \xi}_2  +  {\tilde \xi}_3  \nonumber \\
\dot {\tilde \xi}_3 &= - \lambda_2  {\tilde \xi}_3  +  {\tilde \xi}_4  \nonumber \\ 
 & \ ... \nonumber \\ 
\dot {\tilde \xi}_r &= - \lambda_{r-1}  {\tilde \xi}_r  +  \Gamma(y)\Delta. 
\end{align}  
However, since $\Gamma(y)$ is not assumed uniformly bounded, then the stability properties of the observer-error dynamics~\eqref{observer_error_wrong} can not be inferred from the sole system~\eqref{observer_error_wrong} alone, which is a feature that one may want to seek, in view of a desired ``modular'' (hence simpler) architecture of the unfolding stability analysis. 

Indeed, aimed at seeking such a 
modular structure for the stability analysis that shall follow, consider then 
system~\eqref{sys1}, and define $\bar \Gamma(y) = \Gamma(y) - \Gamma(0)$, so that $\bar \Gamma(0)=0$.
Since $\bar \Gamma(y)$ is smooth and such that $\bar \Gamma(0)=0$, then, by the Mean Value Theorem, it is possible to write $\bar \Gamma(y) =\bar \Gamma_0(y) y$, with $\bar \Gamma_0(y)$ smooth. 
Then define $\bar G (y, \mu) = G (y, \mu) + \bar g(\mu) \bar \Gamma_0(y)\Delta$, and $u = v\frac{\Gamma(0)}{\Gamma(y)}$ (where the ratio $\frac{\Gamma(0)}{\Gamma(y)}$ is well-defined, since $\Gamma(y) \geq \Gamma_0 >0$).  As a result of the previous derivations, it is seen that 
system~\eqref{sys1} can now be rewritten as:
\begin{align}\label{sys1_rewritten}
\dot x &= F(\mu)x + \bar G(y, \mu, \Delta)y + \bar g(\mu) \Gamma (0)(v+\Delta)    \nonumber \\
\dot y &= H(\mu)x + K(y, \mu)y.  
\end{align}
Because $\Delta$ is uniformly bounded within $\bar G(y, \mu, \Delta)$, then identical arguments as the ones in~\cite[ch. 11.3]{IsidoriII} can be employed to show that systems fulfilling Properties~\ref{sys_dynamics}-\ref{min_phase_assumption}-\ref{relative_degree}-\ref{measures}, fulfill
the following property as well:
\begin{property}\label{Dyn_extension}
Consider the ``dynamic extension''
\begin{align}\label{extension}
\dot \xi_2 &= - \lambda_1  \xi_2  +  \xi_3  \nonumber \\
\dot \xi_3 &= - \lambda_2  \xi_3  +  \xi_4  \nonumber \\ 
 & \ ... \nonumber \\ 
\dot \xi_{r} &= - \lambda_{r-1}  \xi_r  +  \Gamma(0)(v + \Delta)
\end{align}
with $\lambda_i >0$, for $i = 1,..., r-1$. Define $\xi =(\xi_2,...,\xi_r)^T$; then 
there exist a constant $\bar{b}_2 >0$, smooth functions $b_2(\mu)$, $D_2(\mu)$ and $d_2(\mu)$, with $b_2(\mu) > \bar{b}_2$, 
and a change of coordinates 
\begin{align}\label{change_of_coordinates}
z &= x - D_2(\mu) \xi - \frac{1}{b_2(\mu)}d_2(\mu)y
\end{align}
such that system~\eqref{sys1_rewritten}-\eqref{extension} - within the new coordinates - can be written as 
 \begin{align}\label{extended_sys1}
\dot z &= \tilde F_2(\mu)z + \tilde G_2(y, \mu, \Delta)y     \nonumber \\
\dot y &= \tilde H_2(\mu)z + \tilde K_2(y, \mu)y + b_2(\mu) \xi_2   \nonumber \\
\dot \xi_2 &= - \lambda_1  \xi_2  +  \xi_3  \nonumber \\
\dot \xi_3 &= - \lambda_2  \xi_3  +  \xi_4  \nonumber \\ 
 & \ ... \nonumber \\ 
\dot \xi_{r} &= - \lambda_{r-1}  \xi_r  +  \Gamma(0)(v + \Delta)
\end{align}
in which $\tilde F_2(\cdot)$, $\tilde G_2(\cdot)$, $\tilde H_2(\cdot)$, $\tilde K_2(\cdot)$ and $b(\cdot)$ are smooth functions of their respective arguments, and $\tilde F_2(\mu)$ is Hurwitz for any $\mu \in {\mathcal P}$.
\end{property}
Because $\Delta$ is unknown, then the dynamic extension~\eqref{extension} can not be implemented. Hence, the following ``open-loop observer'' for~\eqref{extension} is employed as a``proxy'' of~\eqref{extension} 
\begin{align}\label{observer}
\dot {\hat \xi}_2 &= - \lambda_1  {\hat \xi}_2  +  {\hat \xi}_3  \nonumber \\
\dot {\hat \xi}_3 &= - \lambda_2  {\hat \xi}_3  +  {\hat \xi}_4  \nonumber \\ 
 & \ ... \nonumber \\ 
\dot {\hat \xi}_r &= - \lambda_{r-1}  {\hat \xi}_r  +  \Gamma(0)v. 
\end{align}
By then defining ${\tilde \xi}_i = \xi_i - {\hat \xi}_i$, the observer-error dynamics can be written as:
\begin{align}\label{observer_error}
\dot {\tilde \xi}_2 &= - \lambda_1  {\tilde \xi}_2  +  {\tilde \xi}_3  \nonumber \\
\dot {\tilde \xi}_3 &= - \lambda_2  {\tilde \xi}_3  +  {\tilde \xi}_4  \nonumber \\ 
 & \ ... \nonumber \\ 
\dot {\tilde \xi}_r &= - \lambda_{r-1}  {\tilde \xi}_r  +  \Gamma(0)\Delta.
 \end{align}
As opposed to system~\eqref{observer_error_wrong}, the dynamics in~\eqref{observer_error} do not contain any potentially unbounded forcing ``disturbance'' (indeed, the function $\Gamma(y)$ in~\eqref{observer_error_wrong} has been replaced by the constant $\Gamma(0)$ in~\eqref{observer_error}); hence the following lemma can now be easily established:
\begin{lemma}\label{steering_obs_error}
The trajectories of system~\eqref{observer_error} are uniformly bounded. Moreover, for any desired $\epsilon >0$, there exists a choice of $\lambda_{i}>0$, for $i=1,..., r-1$, such that, there exists a time $\bar t >0$, such that, $\abs{{\tilde \xi}_i} < \epsilon$, for $i=2,...,r$, for any $t> \bar t$.
\end{lemma}
\begin{proof}
The result is self-evident and follows directly from the structure of the dynamics in~\eqref{observer_error}.
\end{proof}

\subsection{Closed-loop stability analysis}

\begin{theorem}\label{stability_C2_theorem}
Consider system~\eqref{sys1}, and assume that Properties~\ref{sys_dynamics}-\ref{min_phase_assumption}-\ref{relative_degree}-\ref{measures} hold.
Then there exists a dynamical controller for system~\eqref{sys1} such that:
\begin{enumerate}
	\item the closed-loop trajectories are globally uniformly bounded \vspace{2mm}
	\item the closed-loop trajectories possess a ultimate bound, which can be made as small as desired.
\end{enumerate}
\end{theorem}

\begin{proof}    
Consider system~\eqref{sys1}, along with the dynamic extension~\eqref{extension} and the change of coordinates~\eqref{change_of_coordinates}. Consider then the transformed system~\eqref{extended_sys1} along with the open-loop observer~\eqref{observer}. Then pick $\epsilon >0$, and, consequently, fix $\lambda_1,...,\lambda_{r-1}$ within the observer~\eqref{observer} as stipulated by Lemma~\eqref{steering_obs_error}. By using $\xi_i= {\tilde \xi}_i + {\hat \xi}_i$, and by recalling that $u = v\frac{\Gamma(0)}{\Gamma(y)}$, system~\eqref{extended_sys1} can be written as  
\begin{align}\label{extended_sys2}
\dot z &= \tilde F_2(\mu)z + \tilde G_2(y, \mu, \Delta)y     \nonumber \\
\dot y &= \tilde H_2(\mu)z + \tilde K_2(y, \mu)y + b_2(\mu) \hat \xi_2 +   b_2(\mu) \tilde \xi_2 \nonumber \\
\dot {\hat \xi}_2 &= - \lambda_1  {\hat \xi}_2 +  {\hat \xi}_3  \nonumber \\
\dot {\hat \xi}_3 &= - \lambda_2  {\hat \xi}_3  +  {\hat \xi}_4  \nonumber \\ 
 & \ ... \nonumber \\ 
\dot {\hat \xi}_r &= - \lambda_{r-1}  {\hat \xi}_r  +  \Gamma(y)u.
\end{align}
Because of the results of Lemma~\ref{steering_obs_error}, the quantity $\tilde \xi_2$ within system~\eqref{extended_sys2} is uniformly bounded, hence system~\eqref{extended_sys2} can be written in the form of system~\eqref{sys_generale}, with $y={\xi}_1$, $b_2(\mu) \tilde \xi_2 = \bar d_1$, and ${\hat \xi}_i =\xi_i$ with $i=2,...,r$. Moreover, since system~\eqref{extended_sys2} fulfills Properties~\ref{sys_dynamics_2}-\ref{Zero_dynamics_sys_dynamics_2}-\ref{Well_defined_control_sign_sys_dynamics_2}-\ref{Feedback_info_sys_dynamics_2} as well, in view of Remark~\ref{remark_compacts}, then it follows 
that the results of Theorem~\eqref{theorem_recursion} apply for system~\eqref{extended_sys2} (thereby dictating the existence of a \textit{static} controller as in~\eqref{control_input_sys_generale} for system~\eqref{extended_sys2}, such that points $(1)$ and $(2)$ of the theorem apply).
This, in turn, immediately implies that there exists a \textit{dynamic} controller for system~\eqref{sys1} such that points (1) and (2) of the theorem apply. In particular, by virtue of the results of Lemma~\ref{theorem1}, Lemma~\ref{theorem2} along with Theorem~\ref{theorem_recursion}, by defining 
\begin{align}\label{Rec3} 
\hat x_1 &=  (y, {\hat \xi_2},...,{\hat \xi_{r-1}})^T
\end{align}
from~\eqref{observer} and~\eqref{control_input_sys_generale}, it is seen that the dynamical controller for system~\eqref{sys1} can be given the form
\begin{align}\label{control_input_final}
 \dot {\hat \xi}_2 &= - \lambda_1  {\hat \xi}_2 +  {\hat \xi}_3  \nonumber \\
\dot {\hat \xi}_3 &= - \lambda_2  {\hat \xi}_3  +  {\hat \xi}_4  \nonumber \\ 
 & \ ... \nonumber \\ 
\dot {\hat \xi}_r &= - \lambda_{r-1}  {\hat \xi}_r  +  \Gamma(y)u \nonumber \\
u  &= -\gamma_2(\hat x_1, {\hat \xi}_r)[{\hat \xi}_r - \gamma_1(\hat x_1)\hat x_1] -\frac{({\hat \xi}_r - \gamma_1(\hat x_1)\hat x_1) p^2}{\abs{{\hat \xi}_r - \gamma_1(\hat x_1)\hat x_1}p + \eta}
\end{align}
with $\gamma_1(\cdot)$ and $\gamma_2(\cdot)$ smooth functions of their arguments, with $p >0$ large enough, and with $\eta >0$ as small as desired. This concludes the proof.

\end{proof}

\begin{remark}
It is noted that, by employing the change of coordinates $\xi_i= {\tilde \xi}_i + {\hat \xi}_i$, the effect of the uniformly bounded disturbance $\Delta$ in system~\eqref{extended_sys1}, is reflected within system~\eqref{extended_sys2} by the presence of the forcing, observer-error variable $\tilde \xi_2$.
The latter having been proven uniformly bounded by Lemma~\ref{steering_obs_error}.

\end{remark}

\section{Examples}
Consider the following nonlinear system
\begin{align} \label{sys1_example1}
\dot z_1 &= z_2 + \mu_1 x_1^2, \nonumber\\
\dot z_2 &= -\mu_2 z_1 - \mu_3 z_2 + \mu_4(1 + x_1^2)x_1, \nonumber\\
\dot x_1 &= \mu_5 x_1^3 + (8 + x_1^2 + \mu_6) x_2 + \mu_8(1+x_1^2) z_1 \nonumber\\
&\quad + \mu_9\sin(x_1) z_2 + d_1, \nonumber\\
\dot x_2 &= x_1 x_2^2 + \bigl(2 + \mu_7 \cos(x_1 x_2)\bigr)(u + \Delta_m) \nonumber\\
&\quad + \mu_{10} x_1 x_2 z_1 + \mu_{11} \cos(x_1 x_2) z_2 + d_2,  
\end{align}
where $x_1$ and $x_2$ are available for feedback, whereas the states $z_1$ and $z_2$ are assumed non-measurable. All uncertain parameters are constant and satisfy $\mu_1, \mu_4 \in [-0.5, 0.5]$, $\mu_2, \mu_3 \in [1, 100]$, $\mu_i, \in [-1, 1]$ for $i=5,...,11$.
The disturbances are selected as $d_1 = 0.2 \sin(5t)$, $\Delta_m = 10 + \cos (3 t)$, 
$d_2 = \sin (10t)$. 
By virtue of the above-mentioned choices, it follows that system~\eqref{sys1_example1} satisfies Properties~\ref{sys_dynamics_2}-\ref{Zero_dynamics_sys_dynamics_2}-\ref{Well_defined_control_sign_sys_dynamics_2}-\ref{Feedback_info_sys_dynamics_2}, thereby making it a well-defined example of a system within $\mathcal{C}_1$.

Hence, by following Theorem~\ref{theorem_recursion}, a controller of the form~\eqref{control_input_sys_generale} is specialized for system~\eqref{sys1_example1} as 
\begin{align}\label{controller_example_1}
u  = -\bar\gamma_2(x_1, x_2)[x_2 - \bar \gamma_1(x_1)x_1] -\frac{(x_2 - \bar \gamma_1(x_1)x_1) p^2}{\abs{x_2 - \bar \gamma_1(x_1) x_1}p + \eta}
\end{align}
with 
\begin{align}\label{controller_example_1_gains}
\bar \gamma_1(x_1) &= - k_1  \nonumber\\
\bar\gamma_2(x_1, x_2) &= k_2\bigl(1 + x_1^6 + x_1^4 + x_1^2 x_2^2\bigr)  
\end{align}
and where $p, \eta, k_1, k_2 > 0$ in~\eqref{controller_example_1}-\eqref{controller_example_1_gains}  are control design parameters to be chosen in conformity with the analysis in Section~\ref{Stability_analysis_C_1}.

In view of the forthcoming discussion, the first term on the right-hand side of the  controller in~\eqref{controller_example_1} shall be referred to as the ``stabilizer'', whereas the second one will be denoted as the ``additional module''. The performance of the controlled class $\mathcal{C}_1$ system~\eqref{sys1_example1} is assessed in computer simulation. 
This is done by comparing the performance of the closed-loop system~\eqref{sys1_example1}–\eqref{controller_example_1}–\eqref{controller_example_1_gains} with the ``stabilizer'' enabled and the ``additional module''  disabled versus the one featuring both the ``stabilizer'' and the “additional module” enabled. The proposed simulation has been carried out by selecting the parameters as indicated in Table~1.
In particular, the initial conditions and the system parameters have been selected identical for both cases. The control gains $k_1$ and $k_2$ are identical for both cases as well, whereas the indicated values of $p$ and $\eta$ are solely applicable with the ``additional module'' on. This is highlighted by displaying such values within brackets in Table~1.

\begin{table}[h]\label{Table1}
\caption{Simulation parameters for example of system in $\mathcal{C}_1$}
\label{tab:controller_gains}
\begin{tabular}{ll}
\hline\hline
Initial conditions & $x_1(0) =0.2, \ x_2(0) = 0.1$ \\
              & $z_1(0) = 0.5, \ z_2(0) = 0.3$ \\[6pt]
System parameters & $\mu_1 =-0.1, \ \mu_2=10, \ \mu_3=10$ \\
                & $\mu_4=0.2, \ \mu_5=0.2, \ \mu_6= 1$ \\
                & $\mu_7=0.5, \ \mu_8=-0.5, \  \mu_9 =0.5$\\  
                &$\mu_{10}=0.2, \mu_{11}=-0.3$ \\ [6pt]
Controller gains & $k_1 = 20,\ k_2 = 1, \ (p = 20),\ (\eta = 0.01)$ \\
\hline\hline
\end{tabular}
\end{table} 
\begin{figure} 
\centering
\includegraphics[width=1.1\linewidth]{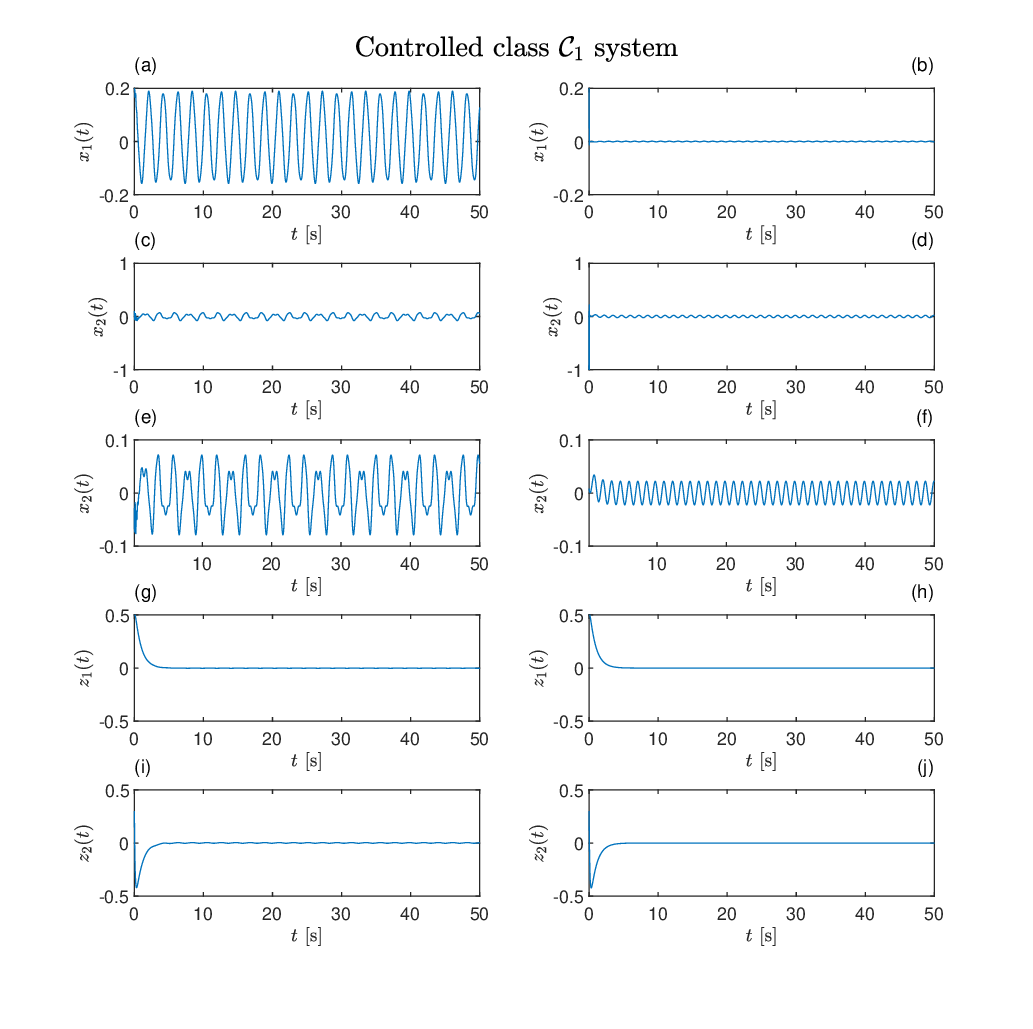}
\caption{Simulation of controlled class $\mathcal{C}_1$ system. The plots on the left feature the closed-loop system performance with the “additional module” disabled in~\eqref{controller_example_1}. The plots on the right show the closed-loop system performance with both the ``stabilizer'' and the “additional module” enabled in~\eqref{controller_example_1}.
In particular, plots (c) and (d) feature the entire time history of the state $x_2$, whereas plots (e) and (f) zoom on its ``steady-state'' evolution.}
\label{plots_class_c1}
\end{figure}

From Fig.~\ref{plots_class_c1} - where the results of the computer simulation are shown - it is seen that both strategies are capable to steer the closed-loop trajectories to a bounded attractor, in spite of incoming forcing disturbances and unknown parameters. Moreover, the strategy featuring both the ``stabilizer'' and the “additional module” enabled in~\eqref{controller_example_1} (plots on the right in Fig.~\ref{plots_class_c1}) shows  a superior ``steady-state'' performance when the states $x_1$ and $x_2$ of system~\eqref{sys1_example1} are considered (compare in particular plots (a) and (e) with plots (b) and (f) in Fig.~\ref{plots_class_c1}).)
Whereas no noticeable difference appears between the two strategies in relation to the ``steady-state'' performance of the states $z_1$ and $z_2$.

Consider now the following nonlinear system
\begin{align}\label{sys1_example2}
\dot{x}_1 &= -3x_1 + x_3 + \mu_1 (y^2 + y^4) \nonumber \\
\dot{x}_2 &= -x_1 - \mu_2 x_2 + \mu_3 y \sin(y) \nonumber \\
\dot{x}_3 &= -2x_3 + \mu_4 \frac{y^2}{1+y^2} + \mu_5 (1 + y^2)\big(u + \Delta(t)\big) \nonumber \\
\dot{y} &= x_1 + \mu_6 (e^{-y^2} + y^2) y
\end{align}
 where the sole $y$ is available for feedback, whereas the states $x_1, x_2, x_3$ are assumed non measurable. All uncertain parameters are constant and satisfy $\mu_1, \mu_4, \mu_6 \in [-5, 5]$, $\mu_2, \mu_3,\mu_5 \in [1, 10]$. The uniformly bounded disturbance is selected as  $\Delta(t) = 1 + 2\sin(3t) + 3\cos(2t) + 3 \cos(16t)$. 
 
 System~\eqref{sys1_example2} is in the form of system~\eqref{sys1}, with 
 \begin{align}\label{matrices_1}
F(\mu) = \begin{pmatrix} -3 & 0 & 1 \\ -1 & -\mu_2 & 0 \\ 0 & 0 & -2 \end{pmatrix}, \quad
    G(y, \mu) = \begin{pmatrix} \mu_1 (y + y^3) \\ \mu_3 \sin(y) \\ \mu_4 \dfrac{y}{1+y^2} \end{pmatrix}  
\end{align}
\begin{align}\label{matrices_2}
 \bar{g}(\mu) = \begin{pmatrix} 0 & 0 & \mu_5 \end{pmatrix}^T,\quad  \Gamma(y) = 1 + y^2, \quad H(\mu) = \begin{pmatrix} 1 & 0 & 0 \end{pmatrix}
\end{align}
\begin{align}\label{matrices_3}
K(y, \mu) = \mu_6 (e^{-y^2} + y^2)
\end{align}
and $\mu = (\mu_1,...,\mu_6)^T$.
By hence considering~\eqref{matrices_1}-\eqref{matrices_2}-\eqref{matrices_3}, and by inspection, it is easily seen that system~\eqref{sys1_example2} fulfills Property~\ref{sys_dynamics}.
Moreover, by standard calculations, it follows that the zero dynamics of system~\eqref{sys1_example2} with respect to $y$ are given by
\begin{align}\label{zero_dyn_ex_2}
 \dot{x}_2 &= -\mu_2 x_2 
\end{align}
which are exponentially stable, since $\mu_2 \in [1, 10]$. Hence system~\eqref{sys1_example2} fulfills Property~\ref{min_phase_assumption} as well. Additionally, after having established (by means of standard arguments) that system~\eqref{sys1_example2} possesses relative degree $r=3$ with respect to $y$, it is also easy to verify - by using the identities in~\eqref{matrices_1}-\eqref{matrices_2} - that
\begin{align}\label{multipl} 
H(\mu) F^{r-2}(\mu) \bar{g}(\mu) = \mu_5. 
\end{align}
Since $\mu_5 \in [1, 10]$, from~\eqref{multipl} it follows that system~\eqref{sys1_example2} fulfills Property~\ref{relative_degree}, too. Finally, by observing that Property~\ref{measures} is fulfilled as well since $y$ is the sole output of system~\eqref{sys1_example2}, it is possible to conclude that system~\eqref{sys1_example2} indeed qualifies as a well-defined example of a system within $\mathcal{C}_2$. 

By hence following Theorem~\ref{stability_C2_theorem}, by defining $\hat{x}_1 = (\hat{\xi}_2, y)^T$,
a control law of the form~\eqref{control_input_final} is specialized for system~\eqref{sys1_example2} as  
\begin{align}\label{controller_example_2}
\dot{\hat{\xi}}_2 &= -\lambda_1 \hat{\xi}_2 + \hat{\xi}_3, \nonumber\\
\dot{\hat{\xi}}_3 &= -\lambda_2 \hat{\xi}_3 + (1+y^2)\,u, \nonumber\\
u &= -\gamma_2(\hat{x}_1,\hat{\xi}_3)\bigl[\hat{\xi}_3 - \gamma_1(\hat{x}_1)\hat{x}_1 \bigr] -\frac{(\hat{\xi}_3 - \gamma_1(\hat{x}_1)\hat{x}_1) p^2}{\abs{\hat{\xi}_3 - \gamma_1(\hat{x}_1)\hat{x}_1}p + \eta}
\end{align}
with 
\begin{align}\label{controller_example_2_gains}
\gamma_1(\hat{x}_1) &= -(k_1,\; k_2) \nonumber\\
\gamma_2(\hat{x}_1,\hat{\xi}_3) &= c(1+y^2)^3  
\end{align}
where $\lambda_1, \lambda_2, k_1, k_2, c, p, \eta >0$ are control design parameters 
to be chosen in conformity with the analysis in Section~\ref{Stability_analysis_C_2}.

 In analogy with the previous case, we shall denote the first term on the right-hand side of the last equation in~\eqref{controller_example_2} as the ``stabilizer'', whereas the second term will be referred to as the ``additional module''. 
Then, the evolution of the controlled class $\mathcal{C}_2$ system~\eqref{sys1_example2} is assessed in computer simulation, by comparing the performance of the closed-loop system~\eqref{sys1_example2}-\eqref{controller_example_2}-\eqref{controller_example_2_gains} with the sole ``stabilizer'' enabled (hence, devoid of the ``additional module'') versus the one featuring both the ``stabilizer'' and the “additional module” enabled.  The proposed simulation has been performed by selecting the parameters as indicated in Table~2. In particular, the initial conditions and the system parameters have been selected identical for both cases. The control gains $\lambda_1$, $\lambda_2$, $k_1$ and $k_2$ are identical for both cases as well, whereas the indicated values of $p$ and $\eta$ are solely applicable with the ``additional module'' on. This is highlighted by displaying such values within brackets in Table~2.
\begin{table}[h]\label{Table1}
\caption{Simulation parameters for example of system in $\mathcal{C}_2$}
\label{tab:controller_gains}
\begin{tabular}{ll}
\hline\hline
Initial conditions & $x_1(0) =0.2, \ x_2(0) = -0.2$ \\
                & $x_3(0) = 0.1, \ y(0) = -1$ \\[6pt]
System parameters & $\mu_1 =-1.2, \ \mu_2=1.3, \ \mu_3=10$ \\
                & $\mu_4=2, \ \mu_5=1.5, \ \mu_6=0.5 $ \\[6pt]
Controller gains & $\lambda_1 = 4,\ \lambda_2 = 4, \ k_1 =45 ,\ k_2 = 2$ \\
                 & $c = 20, \ (p = 25),\ (\eta = 0.01)$ \\
\hline\hline
\end{tabular}
\end{table} 

The results of the computer simulation are shown in Fig.~\ref{plots_class_c2}. From the latter 
it is seen that both strategies are capable to steer the closed-loop trajectories to a bounded attractor, in spite of incoming forcing disturbances and unknown parameters.
Moreover, the strategy featuring both the ``stabilizer'' and the “additional module” enabled in~\eqref{controller_example_2} (plots on the right in Fig.~\ref{plots_class_c2}) shows a superior ``steady-state'' performance with respect to all of the states of system~\eqref{sys1_example2}.
\begin{figure} 
\centering
\includegraphics[width=1.1\linewidth]{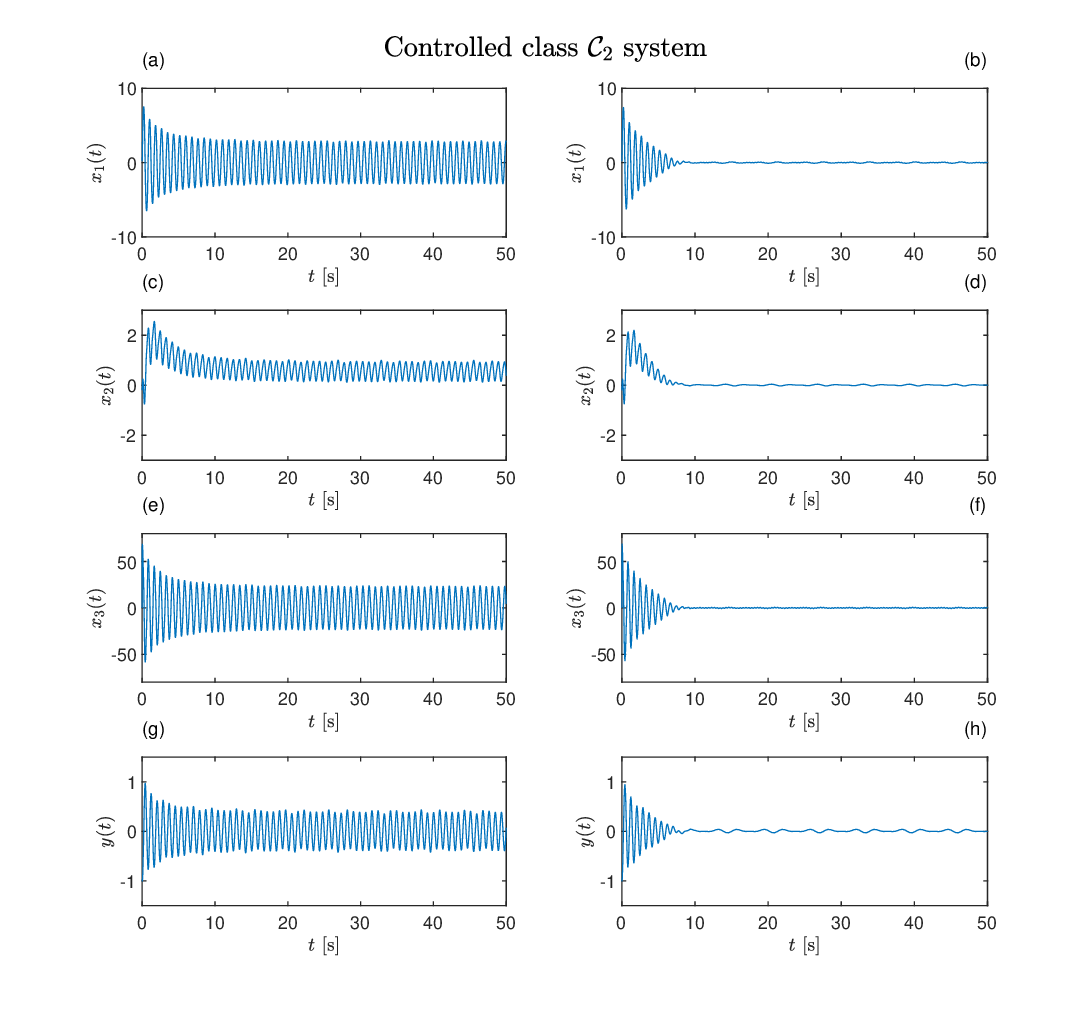}
\caption{Simulation of controlled class $\mathcal{C}_2$ system. The plots on the left feature the closed-loop system performance with the “additional module” disabled in~\eqref{controller_example_2}. The plots on the right show the closed-loop system performance with both the ``stabilizer'' and the “additional module” enabled in~\eqref{controller_example_2}.}
\label{plots_class_c2}
\end{figure}

\section{Conclusions}
In this work, the problem of non-adaptive global robust disturbance rejection has been addressed for two distinct classes of nonlinear systems, respectively denoted by $\mathcal{C}_1$ and $\mathcal{C}_2$.
For systems within $\mathcal{C}_1$ (i.e., nonlinear systems in strict-feedback form, with linear and Hurwitz zero-dynamics with unmeasurable states, and with forcing, unmatched, additive disturbances),  a high-gain-based control architecture was shown to achieve closed-loop input-to-state stability and global asymptotic convergence to an arbitrarily small attractor. A supplementary sliding-mode unit was then embedded within the control architecture, to improve disturbance rejection while potentially lowering the required high-gain expenditure. For systems within $\mathcal{C}_2$ (i.e, minimum-phase, uncertain, nonlinear systems with relative degree greater than one, featuring possibly unbounded, with possibly unbounded derivatives, output-dependent nonlinearities, with matched additive forcing disturbances), an open-loop observer replaced the classic dynamic extension (which became infeasible under unknown disturbances), and the results derived for $\mathcal{C}_1$ were adapted to design an output-feedback dynamic controller that ensured global uniform boundedness and asymptotic regulation to an arbitrarily small attractor. The provided stability proofs are inherently inductive, thereby circumventing the curse of dimensionality often encountered in strict-feedback designs.

\bibliographystyle{unsrt}        
\bibliography{Biblio_Horrebant}           



\end{document}